\documentclass[12pt]{article}
\usepackage{amssymb,amstext,amsmath,amsthm,amscd,enumerate,color}
\usepackage[T1]{fontenc}
 \usepackage{dsfont}

\theoremstyle{definition}\newtheorem{definition}{Definition}[section]
\theoremstyle{plain}\newtheorem{theorem}[definition]{Theorem}
\theoremstyle{plain}\newtheorem{proposition}[definition]{Proposition}
\theoremstyle{plain}
\theoremstyle{plain}\newtheorem{lemma}[definition]{Lemma}
\theoremstyle{definition}\newtheorem{assumption}[definition]{Assumption}
\theoremstyle{definition}\newtheorem{example}[definition]{Example}
\theoremstyle{definition}\newtheorem{remark}[definition]{Remark}

\numberwithin{equation}{section}

\newcommand{\one}{\mathds{1}}
\newcommand{\M}{\mathcal{M}}
\newcommand{\N}{\mathbb{N}}
\newcommand{\cO}{\mathcal{O}}

\newcommand{\1}{{\ell^1(\N)}}
\newcommand{\2}{{\ell^2(\N)}}
\newcommand{\3}{{\ell^\infty(\N)}}

\newcommand{\xad}{x_\alpha^\delta}

\newcommand{\xdag}{x^\dagger}

\newcommand{\la}{\langle}
\newcommand{\ra}{\rangle}
\newcommand{\CC}{{\mathbb C}}
\newcommand{\NN}{{\mathbb N}}
\newcommand{\TT}{\mathbb{T}}
\newcommand{\ZZ}{\mathbb{Z}}
\newcommand{\cW}{\mathcal{W}}

\DeclareMathOperator*{\sgn}{sgn}
\DeclareMathOperator*{\supp}{supp}

\begin{document}

\date{\today}

\title{A unified approach to convergence rates for $\ell^1$-regularization and lacking sparsity}

\author{
{\sc Jens Flemming}\,,
{\sc Bernd Hofmann}\,,
{\sc Ivan Veseli\'c}\,
\thanks{Faculty of Mathematics, Technische Universit\"at Chemnitz, 09107 Chemnitz, Germany.}}

\maketitle

\begin{abstract}
In $\ell^1$-regularization, which is an important tool in signal and image processing, one usually is concerned with signals
and images having a sparse representation in some suitable basis, e.g.~in a wavelet basis.
Many results on convergence and convergence rates of sparse approximate solutions to linear ill-posed problems are known, but rate results
for the $\ell^1$-regularization in case of lacking sparsity had not been published until 2013.
In the last two years, however, two articles appeared providing sufficient conditions for convergence rates in case of non-sparse
but almost sparse solutions. In the present paper, we suggest a third sufficient condition, which unifies the existing two and,
by the way, also incorporates the well-known restricted isometry property.
\end{abstract}

\noindent\textbf{MSC2010 subject classification:} 65J20, 47A52, 49N45

\medskip


\section{Introduction}

The method of $\ell^1$-regularization as a variety of Tikhonov-type regularization became of significant interest in the past ten years, mainly driven by its applications
in signal processing and imaging. We consider a bounded linear operator $A:\1\rightarrow Y$ mapping from $\1$ into a Banach space $Y$. The idea is to solve with some exponent $1 \le p < \infty$ the minimization problem
\begin{equation}\label{eq:l1}
\|Ax-y^\delta\|_Y^p+\alpha\|x\|_\1\to\min, \qquad \mbox{subject to} \quad x=(x_k)_{k=1}^\infty\in\1,
\end{equation}
with minimizers $\xad$ as an auxiliary problem for finding a solution $\xdag=(\xdag_k)_{k=1}^\infty$ of the linear operator equation
\begin{equation} \label{eq:opeq}
Ax=y,\quad x\in \1,\;y \in Y,
\end{equation}
with  exact right-hand side $y$ from the range $\mathcal{R}(A)$ of the operator $A$ and for available data $y^\delta \in Y$ which are characterized by the deterministic noise model
\begin{equation}
\|y-y^\delta\|_Y\leq\delta
\end{equation}
with noise level $\delta\geq 0$. The regularization parameter $\alpha>0$ in \eqref{eq:l1} controls the influence of the penalty term $\|x\|_\1$, hence the trade-off between
data fidelity and approximation. For details on the theoretical background of $\ell^1$-regularization we refer to \cite{BeBu11,BreLor09,Grasm10,Grasmei11,Lorenz08,RamRes10}.

The problems (\ref{eq:opeq}) and (\ref{eq:l1}) are motivated by the assumption that for a separable infinite dimensional Hilbert space $X$ with inner product
$\langle \cdot,\cdot\rangle_X$, some orthonormal basis $\{u^{(k)}\}_{k \in \N}$ and for an injective and bounded
linear operator $\widetilde A: X \to Y$ the solutions of the operator equation
\begin{equation} \label{eq:opeqtilde}
\widetilde A\, \widetilde x =y,\quad \widetilde x\in X,\;y \in Y,
\end{equation}
are sparse in the sense that only a finite number of basis elements $u^{(k)}$ becomes active in the solution. This means that the associated sequence $x=(x_k)_{k=1}^\infty:=(\langle \widetilde x,u^{(k)}\rangle_X)_{k=1}^\infty$ of Fourier coefficients for the solution to (\ref{eq:opeqtilde}) is in $\ell^0(\N)$. Taking into account the restriction to $\1$ of the synthesis operator $$L: x \in \1  \mapsto \sum_{k=1}^\infty x_k u^{(k)} \in X$$ and setting $A:=\widetilde A \circ L$
it makes sense to solve (\ref{eq:l1}) due to $\xad \in \ell^0(\N)$, i.e., because the minimizers of \eqref{eq:l1} (which always exist - not necessarily unique - and depend stably on the data $y^\delta$) are sparse
(cf.~\cite[Proposition~2.8]{BurFleHof13}).
Since the synthesis operator $L: \1 \to X$ is injective and bounded, the linear operator $A: \1 \to Y$ is also injective and bounded if $\widetilde A: X \to Y$ is.
In \cite[Proposition~4.6]{FHV15} we have shown that the operator equation \eqref{eq:l1} with $A=\widetilde A \circ L$ as introduced above is always ill-posed, i.e.~$\mathcal{R}(A) \not= \overline{\mathcal{R}(A)}$, and even ill-posed of type~II in the sense of Nashed (cf.~\cite{Nashed86}). This is an additional motivation for the use of
a regularization method like \eqref{eq:l1} for the stable approximate solution of equation (\ref{eq:opeq}).

In \cite{BurFleHof13} and \cite{FleHeg14} for linear problems \eqref{eq:opeq}, as well as in \cite{BoHo13} for nonlinear problems, it was shown that the $\ell^1$-regularization based on \eqref{eq:l1} is also applicable if the exact solution $\xdag \in \1$ to (\ref{eq:opeq}), which is unique for injective $A$, is not sparse but \emph{almost sparse}. This means that components decay fast enough. In fact, the assumption $x^\dagger\in\1$ suffices
to obtain convergence rates for the convergence of the minimizers $x_\alpha^\delta$ to the exact solution $x^\dagger$ if the
unit sequences
$$e^{(k)}:=(0,\ldots,0,1,0,\ldots)\quad\text{(with $1$ at position $k$)}$$
and the operator $A$ fit together well.
The faster the components $|\xdag_k|$ of $\xdag$ decay to zero as $k \to \infty$, the faster is the convergence.

The aim of the present paper is to unify the assumptions used in \cite{BurFleHof13} and \cite{FleHeg14} to prove convergence rates.
As a byproduct also the so called \emph{restricted isometry property} appearing frequently in connection with the reconstruction
of sparse signals (see, e.g., \cite{CanRomTao06}) will be covered by our unified setting.

The structure of the paper is as follows: In Section~\ref{sc:suff} we present the new setting and prove convergence rates on its basis.
Sections~\ref{sc:smooth} and \ref{sc:nonsmooth} relate the approaches from \cite{BurFleHof13} and \cite{FleHeg14}, respectively,
to the new setting and Section~\ref{sc:rip} incorporates the restricted isometry property.
Moreover, Section~\ref{sc:examples} provides two examples and Section~\ref{sc:outlook} completes the paper with some outlook.

\section{Sufficient condition for convergence rates}\label{sc:suff}

To prove convergence rates for $\ell^1$-regularization we use the powerful technique of variational inequalites (also known as
variational source conditions or variational smoothness assumptions), that is, we show that there are a constant $\beta\in(0,1]$ and
a concave index function $\varphi:[0,\infty)\rightarrow[0,\infty)$ such that
\begin{equation}\label{eq:vi}
\beta\|x-x^\dagger\|_\1\leq\|x\|_\1-\|x^\dagger\|_\1+\varphi(\|Ax-Ax^\dagger\|_Y)\quad\text{for all $x\in\1$}.
\end{equation}
Here, a function $\varphi$ is called index function if it is continuous and strictly increasing with $\varphi(0)=0$. Note that a variational inequality (\ref{eq:vi})
is closely connected with the injectivity of $A$ (cf.~\cite[Proposition~5.6]{BurFleHof13}). Therefore we restrict our considerations in this paper to the case of injective linear operators $A$.
Based on (\ref{eq:vi}) convergence rates
$$\|x_\alpha^\delta-x^\dagger\|_\1=\cO(\varphi(\delta))\quad\text{if $\delta\to 0$},$$
can be shown without any further assumptions, where $\alpha=\alpha(\delta,y^\delta)$ has to be chosen in one of several possible ways.
Details on the technique of variational inequalities for obtaining convergence rates in Tikhonov-type regularization can be found in \cite{Andreev14,AnzHofMat14,BoHo10,Flemmingbuch12,Grasm10,HKPS07,HofMat12,WeHo12}.

Before we come to the convergence rate result, we introduce some notation:
By $\sgn x$ for $x\in\1$ we denote the $\ell^\infty$-sequence defined by
$$[\sgn x]_k:=\begin{cases}1,&\text{if $x_k>0$},\\0,&\text{if $x_k=0$},\\-1,&\text{if $x_k<0$}.\end{cases}$$
For index sets $M\subseteq\N$ we define
$$\one_M:=\{\xi\in\3:\supp\xi\subseteq M\text{ and }\xi_k\in\{-1,0,1\}\text{ for all }k\in\N\}$$
and the projections $P_M:\3\rightarrow\3$ and $P_M:\1\rightarrow\1$ by
$$[P_M\xi]_k:=\begin{cases}\xi_k,&\text{if $k\in M$,}\\0,&\text{else.}\end{cases}$$

The following set of assumptions unifies the conditions used in \cite{BurFleHof13} and \cite{FleHeg14} for obtaining
convergence rates.

\begin{assumption}\label{as:source}
For each $n\in\N$ let $\M_n$ be a family of index sets $M\subseteq\N$ with
cardinality $\vert M\vert=n$ and assume that $\{1,\ldots,n\}\in\M_n$.
Then for each $n\in\N$ and each
\begin{equation}\label{eq:cupm}
\xi\in\bigcup_{M\in\M_n}\one_M
\end{equation}
we assume that there is some $\eta=\eta(n,\xi)\in Y^\ast$ with the following properties:
\begin{itemize}
\item[(i)]
$\|\eta\|_{Y^\ast}\leq\gamma_n$ for some constant $\gamma_n>0$ independent of $\xi$,
\item[(ii)]
$P_{\supp\xi}A^\ast\eta=\xi$,
\item[(iii)]
$\|(I-P_{\supp\xi})A^\ast\eta\|_\3\leq c$ for some $c\in[0,1)$ independent of $\xi$ and $n$.
\end{itemize}
\end{assumption}

Typical choices of the families $\M_n$ are the extreme cases
$$\M_n=\{\{1,\ldots,n\}\}\quad\text{and}\quad\M_n=\{M\subseteq\N:\vert M\vert=n\}.$$
In the first case we have
$$\bigcup_{M\in\M_n}\one_M=\{\xi\in\3:\supp\xi\subseteq \{1,\ldots,n\}\text{ and }\xi\in\one_\N\}$$
in \eqref{eq:cupm} and in the second case
$$\bigcup_{M\in\M_n}\one_M=\{\xi\in\3:\vert\supp\xi\vert\leq n\text{ and }\xi\in\one_\N\}.$$

\begin{theorem}\label{th:vi}
Let Assumption~\ref{as:source} be satisfied for some sequence $(\M_n)_{n\in\N}$ of families of index sets
and assume that the corresponding sequence $(\gamma_n)_{n\in\N}$ in (i) of Assumption~\ref{as:source} contains an element $\gamma_k$
such that $\gamma_k<\inf_{l>k}\gamma_l$.
Then a variational inequality~\eqref{eq:vi} holds with
\begin{equation}\label{eq:phi}
\varphi(t)=2\inf_{n\in\N}\left(\inf_{M\in\M_n}\|(I-P_M)x^\dagger\|_\1+\frac{1}{1+c}\,\gamma_n\,t\right),
\end{equation}
where $\varphi$ is a concave index function and $c$ is taken from (iii) in Assumption~\ref{as:source}.
The constant $\beta$ in \eqref{eq:vi} attains the form $\beta=\frac{1-c}{1+c}$.
\end{theorem}

\begin{remark}
The assumption that $\gamma_k<\inf_{l>k}\gamma_l$ for some $k$ is no relevant restriction since the $\gamma_n$ are
upper bounds for the source elements $\eta$ in Assumption~\ref{as:source}~(i).
We could choose a strictly increasing sequence here, for example.
Without this additional assumption it might happen that $\varphi$ is not strictly increasing.
Nonetheless, $\varphi$ would still remain nondecreasing.
\par
On the other hand, in \cite[Theorem~4.11]{Flemmingbuch12} it was shown that also variational inequalities with nondecreasing but not
necessarily strictly increasing function $\varphi$ yield convergence rates as long as $\varphi$ is strictly increasing
in a small neighborhood of zero. This is indeed the case here without additional assumptions on the behaviour of the
$\gamma_n$. It suffices to show that there is some $t>0$ with $\varphi(t)>0$, which is a consequence of Assumption~\ref{as:source}
and the boundedness of $A^\ast$. Then concavity of $\varphi$ implies the assertion.
\end{remark}

\begin{proof}[Proof of Theorem~\ref{th:vi}]
Set $\beta=\frac{1-c}{1+c}$ with $c\in[0,1)$ from (iii) in Assumption~\ref{as:source}
and fix $x\in\1$, $n\in\N$, $M\in\M_n$.
For
$$\xi:=P_M\sgn(x-x^\dagger)\in\bigcup_{\tilde{M}\in\M_n}\one_{\tilde{M}}$$
Assumption~\ref{as:source} provides $\eta\in Y^\ast$ and $\gamma_n>0$ such that
$$\|\eta\|_{Y^*}\leq\gamma_n,\quad\xi=P_{\supp\xi}A^\ast\eta,\quad\|(I-P_{\supp\xi})A^\ast\eta\|_\3\leq c.$$
Thus,
\begin{align*}
\|P_M(x-x^\dagger)\|_\1
&=\la\xi,x-x^\dagger\ra_{\3 \times \1}
=\la P_{\supp\xi}A^\ast\eta,x-x^\dagger\ra_{\3 \times \1}\\
&=\la P_{\supp\xi}A^\ast\eta-A^\ast\eta,x-x^\dagger\ra_{\3 \times \1}+\la A^\ast\eta,x-x^\dagger\ra_{\3 \times \1}\\
&=-\bigl\la(I-P_{\supp\xi})A^\ast\eta,(I-P_M)(x-x^\dagger)\bigr\ra_{\3 \times \1}+\la\eta,Ax-Ax^\dagger\ra_{Y^* \times Y}\\
&\leq c\|(I-P_M)(x-x^\dagger)\|_\1+\gamma_n\|Ax-Ax^\dagger\|_Y
\end{align*}
and the triangle inequality yields
\begin{equation}\label{eq:pm}
\|P_M(x-x^\dagger)\|_\1\leq c\left(\|(I-P_M)x\|_\1+\|(I-P_M)x^\dagger\|_\1\right)+\gamma_n\|Ax-Ax^\dagger\|_Y.
\end{equation}
Now
\begin{align*}
\lefteqn{\beta\|x-x^\dagger\|_\1-\|x\|_\1+\|x^\dagger\|_\1}\\
&=\beta\|P_M(x-x^\dagger)\|_\1+\beta\|(I-P_M)(x-x^\dagger)\|_\1-\|P_Mx\|_\1-\|(I-P_M)x\|_\1\\
&\quad+\|P_Mx^\dagger\|_\1+\|(I-P_M)x^\dagger\|_\1
\end{align*}
together with
$$\beta\|(I-P_M)(x-x^\dagger)\|_\1\leq\beta\|(I-P_M)x\|_\1+\beta\|(I-P_M)x^\dagger\|_\1$$
and
$$\|P_Mx^\dagger\|_\1=\|P_M(x-x^\dagger-x)\|_\1\leq\|P_M(x-x^\dagger)\|_\1+\|P_Mx\|_\1$$
shows
\begin{align*}
\lefteqn{\beta\|x-x^\dagger\|_\1-\|x\|_\1+\|x^\dagger\|_\1}\\
&\qquad\leq 2\|(I-P_M)x^\dagger\|_\1+(1+\beta)\|P_M(x-x^\dagger)\|_\1\\
&\qquad\quad\,-(1-\beta)\bigl(\|(I-P_M)x\|_\1+\|(I-P_M)x^\dagger\|_\1\bigr).
\end{align*}
Combining this estimate with the previous estimate \eqref{eq:pm} and taking into account that $\beta=\frac{1-c}{1+c}$
and $c=\frac{1-\beta}{1+\beta}$ we obtain
$$\beta\|x-x^\dagger\|_\1-\|x\|_\1+\|x^\dagger\|_\1
\leq 2\|(I-P_M)x^\dagger\|_\1+\frac{2}{1+c}\,\gamma_n\|Ax-Ax^\dagger\|_Y.$$
Taking the infimum over all $M\in\M_n$ and then over all $n\in\N$ proves that $\varphi$ has the form \eqref{eq:phi}.
\par
It remains to show that $\varphi$ is a concave index function.
At first we note that $\varphi$ is concave and upper semicontinuous as the infimum of affine functions.
This implies continuity on the interior $(0,\infty)$ of its domain.
By Assumption~\ref{as:source} we know that $\{1,\ldots,n\}\in\M_n$. Therefore
$$\varphi(0)=2\inf_{n\in\N}\inf_{M\in\M_n}\|(I-P_M)x^\dagger\|_\1
\leq 2\inf_{n\in\N}\sum_{k=n+1}^\infty\vert x^\dagger_k\vert=0.$$
Together with the nonnegativity and the upper semicontinuity this yields continuity of $\varphi$ at $0$
and therefore on the whole domain $[0,\infty)$.
\par
To show that $\varphi$ is strictly increasing we first note that due to concavity it suffices
to show $\varphi(t_1)<\varphi(t_2)<\ldots$ for some strictly increasing sequence $(t_m)_{m\in\N}$ with $t_m\to\infty$ if $m\to\infty$.
Set $c_n:=\inf_{M\in\M_n}\|(I-P_M)x^\dagger\|_\1$ for $n\in\N$ and choose $k$ such that $\gamma_k<\inf_{l>k}\gamma_l$.
If we choose the sequence $(t_m)_{m\in\N}$ such that for each $m \in \NN$
$$t_m\geq(1+c)\,\frac{c_k}{\gamma_n-\gamma_k}\quad\text{for all $n>k$},$$
then we see
$$c_k+\frac{1}{1+c}\,\gamma_k\,t_m\leq c_n+\frac{1}{1+c}\,\gamma_n\,t_m\quad\text{for all $n>k$}.$$
Thus, the infimum over $n\in\N$ in \eqref{eq:phi} with $t=t_m$ is attained at some $n_m\leq k$.
Consequently,
$$\varphi(t_{m-1})\leq 2c_{n_m}+\frac{2}{1+c}\,\gamma_{n_m}\,t_{m-1}<2c_{n_m}+\frac{2}{1+c}\,\gamma_{n_m}\,t_m=\varphi(t_m)$$
for all $m\geq 2$.
\end{proof}

\begin{remark}
If we choose $\M_n=\{\{1,\ldots,n\}\}$ then the inner infimum in \eqref{eq:phi} becomes
$$\inf_{M\in\M_n}\|(I-P_M)x^\dagger\|_\1=\sum_{k=n+1}^\infty\vert x^\dagger_k\vert,$$
which is a way to describe the decay of the components of the sequence $x^\dagger$.
Another way shows up if we choose $\M_n=\{M\subseteq\N:\vert M\vert=n\}$. Then
$$\inf_{M\in\M_n}\|(I-P_M)x^\dagger\|_\1=\sum_{k=n+1}^\infty\bigl\vert x^\dagger_{\kappa(k)}\bigr\vert$$
where $\kappa:\N\rightarrow\N$ is a reordering of $\N$ such that $x^\dagger_{\kappa(k)}\geq x^\dagger_{\kappa(k+1)}$
for all $k\in\N$.
\end{remark}

\section{Smooth basis}\label{sc:smooth}

In \cite{BurFleHof13} it was shown that convergence rates for $\ell^1$-regularization can be obtained even in case
of nonsparse solutions $x^\dagger$ if $\{e^{(k)}\}_{k\in\N}$, which is a Schauder basis in $c_0$ and in $\ell^q(\N)$ for all $1 \le q <\infty$, satisfies the series of range conditions
\begin{equation}\label{eq:series}
e^{(k)}=A^\ast f^{(k)} \quad \mbox{for some} \quad f^{(k)}\in Y^\ast \quad \mbox{and\,\,\, for all} \; k \in \N.
\end{equation}
We mention that we always have weak convergence
\begin{equation} \label{eq:weakA}
Ae^{(k)} \rightharpoonup 0 \quad \mbox{in $Y$},
\end{equation}
since $\{e^{(k)}\}_{k \in \N}$
is an orthonormal basis in $\2$ with $e^{(k)} \rightharpoonup 0$ in $\2$ and \linebreak $\widetilde A: X \to Y$ is a continuous linear operator.
Then it holds $\mathcal{R}(A^*) \subset c_0$ (cf.~\cite[Remark 2.5]{BurFleHof13}). If the operator $\widetilde A$ and hence $A$ are compact operators, then this is sufficient - but not necessary - to have even
\begin{equation} \label{eq:strongA}
\lim \limits_{k \to \infty} \|Ae^{(k)}\|_Y =0,
\end{equation}
and we refer to \cite{FHV15} regarding the distinction between the two cases (\ref{eq:weakA}) and (\ref{eq:strongA}). Under (\ref{eq:strongA}) the norms  $\|f^{(k)}\|_{Y^*}$ of the source elements  $f^{(k)}$
in (\ref{eq:series}) tend to infinity as $k \to \infty$.
It is not difficult to see that (\ref{eq:series}) can be interpreted as some kind of smoothness of the orthonormal basis $\{u^{(k)}\}_{k \in \N}$ in the separable Hilbert space $X$,
and we refer to \cite{AnzHofRam13} for several examples.
On the other hand, as the following proposition shows it turns out that  condition (\ref{eq:series}) is a special case of Assumption~\ref{as:source} in Section~\ref{sc:suff}.

\begin{proposition}
Assumption~\ref{as:source} with $\M_n=\{\{1,\ldots,n\}\}$ and $c=0$ in (iii) is satisfied if and only if the condition (\ref{eq:series}) is satisfied. The constants $\gamma_n$ in (i) of
Assumption~\ref{as:source} can be chosen as $\gamma_n=\sum_{k=1}^n\|f^{(k)}\|_{Y^*}$.
\end{proposition}

\begin{proof}
With $c=0$ in Assumption~\ref{as:source} items (ii) and (iii) together are equivalent to $\xi=A^\ast\eta$.
To proof necessity of basis smoothness choose $\xi=e^{(n)}$ in Assumption~\ref{as:source}.
To see sufficiency we write
$$\xi=\sum_{k=1}^n(\sgn\xi_k)e^{(k)}=\sum_{k=1}^n(\sgn\xi_k)A^\ast f^{(k)}=A^\ast\sum_{k=1}^n(\sgn\xi_k)f^{(k)}$$
and estimate
$$\left\|\sum_{k=1}^n(\sgn\xi_k)f^{(k)}\right\|_{Y^*}\leq\sum_{k=1}^n\|f^{(k)}\|_{Y^*}.$$
That is, Assumption~\ref{as:source} is satisfied with $\eta=\sum_{k=1}^n(\sgn\xi_k)f^{(k)}$.
\end{proof}

\section{Nonsmooth basis}\label{sc:nonsmooth}

In \cite{FleHeg14} it was shown that there are relevant examples where the operator $A$ does not satisfy the condition (\ref{eq:series}), which means that the corresponding orthonormal basis of the
underlying Hilbert space is not smooth enough. Consequently, a weaker condition also yiedling convergence rates has been formulated there.
The next proposition states that the assumptions of \cite{FleHeg14} are covered by Assumption~\ref{as:source}.

\begin{proposition}
Assumption~\ref{as:source} with $\M_n=\{\{1,\ldots,n\}\}$ is satisfied if for each $n\in\N$ and all $k\in\{1,\ldots,n\}$
there are $f^{(n,k)}\in Y^\ast$ such that
\begin{itemize}
\item[(a)]
for each $k\in\{1,\ldots,n\}$ we have $[A^\ast f^{(n,k)}]_l=0$ for $l\in\{1,\ldots,n\}\setminus k$ and
$[A^\ast f^{(n,k)}]_k=1$,
\item[(b)]
$\sum\limits_{k=1}^n\bigl\vert[A^\ast f^{(n,k)}]_l\bigr\vert\leq c$ for all $l>n$ and some $c\in[0,1)$.
\end{itemize}
The constants $c$ here and in Assumption~\ref{as:source}~(iii) coincide. The constants $\gamma_n$ in
Assumption~\ref{as:source}~(i) can be chosen as $\gamma_n=\sum_{k=1}^n\|f^{(n,k)}\|_{Y^*}$.
\end{proposition}

\begin{proof}
Assume that there are $f^{(n,k)}$ as in the proposition. Then for $n\in\N$ and $\xi$ as in Assumption~\ref{as:source} we have
\begin{align*}
\xi
&=\sum_{k\in\supp\xi}\xi_ke^{(k)}
=\sum_{k\in\supp\xi}\xi_k\left(A^\ast f^{(n,k)}-(I-P_{\supp\xi})A^\ast f^{(n,k)}\right)\\
&=A^\ast\sum_{k\in\supp\xi}\xi_kf^{(n,k)}-(I-P_{\supp\xi})A^\ast\sum_{k\in\supp\xi}\xi_kf^{(n,k)}.
\end{align*}
With
$$\eta:=\sum_{k\in\supp\xi}\xi_kf^{(n,k)}$$
this on the one hand yields
$$\xi=P_{\supp\xi}\xi=P_{\supp\xi}A^\ast\eta-P_{\supp\xi}(I-P_{\supp\xi})A^\ast\eta=P_{\supp\xi}A^\ast\eta$$
and on the other hand
\begin{align*}
\|(I-P_{\supp\xi})A^\ast\eta\|_\3
&=\left\|\sum_{k\in\supp\xi}\xi_k(I-P_{\{1,\ldots,n\}})A^\ast f^{(n,k)}\right\|_\3\\
&=\sup_{l>n}\sum_{k\in\supp\xi}\xi_k\bigl[A^\ast f^{(n,k)}\bigr]_l
\leq\sup_{l>n}\sum_{k=1}^n\bigl\vert\bigl[A^\ast f^{(n,k)}\bigr]_l\bigr\vert\\
&\leq c.
\end{align*}
Thus, Assumption~\ref{as:source} is satisfied with
$$\gamma_n=\|\eta\|_{Y^*}\leq\sum_{k=1}^n\|f^{(n,k)}\|_{Y^*}.$$
\end{proof}

\begin{remark}
In contrast to the case of basis smoothness we now do not have full equivalence of Assumption~\ref{as:source} and
the assumptions in \cite{FleHeg14}. But the interested reader may inspect the assumptions in detail and find that the difference
between both concepts is only a tiny technicality and in no way substantial.
\end{remark}

\section{Restricted isometry property}\label{sc:rip}

In \cite{CanRomTao06} (see also \cite[Section~5]{Grasmei11}) the condition
\begin{multline}\label{eq:uup}
\forall \, n \in \NN \; \exists\, \zeta_n, \gamma_n \in (0,\infty) \text{ such that } \zeta_n\|x\|_\1\geq\|Ax\|_Y\geq\frac{1}{\gamma_n}\|x\|_\1
\\
\quad\text{for all $x\in\1$ with $\vert\supp x\vert\leq n$}
\end{multline}
is used to prove error estimates for the recovery of sparse signals, where $|\supp x|$ denotes the number of nonzero components in the infinite sequence $x=(x_k)_{k=1}^\infty$.
This condition is known as \emph{restricted isometry property} or \emph{uniform uncertainty principle}.
In the following we will elucidate the relation of the lower bound
\begin{equation}\label{eq:rip}
\|Ax\|_Y\geq\frac{1}{\gamma_n}\|x\|_\1\quad\text{for all $x\in\1$ with $\vert\supp x\vert\leq n$}
\end{equation}
to the results obtained in the previous sections. To distinguish it from the (twosided) uniform uncertainty principle,
we will call it \emph{restricted injectivity property}.
The following proposition shows a tight connection to Assumption~\ref{as:source}.

\begin{proposition} \label{pro:nearequi}
The items (i) and (ii) of Assumption~\ref{as:source} with $$\M_n=\{M\subseteq\N:\vert M\vert\leq n\}$$ are satisfied if and
only if condition~\eqref{eq:rip} holds. The constants $\gamma_n$ coincide in both conditions for all $n \in \N$.
\end{proposition}

\begin{proof}
Let \eqref{eq:rip} be satisfied and let $x\in\1$. Then for each $n\in\N$ and each fixed $\xi$ as in Assumption~\ref{as:source} we can estimate
\begin{align*}
\vert\la\xi,x\ra_{\3 \times \1}\vert
&=\left\vert\sum_{k\in\supp\xi}\xi_k x_k\right\vert
\leq \sum_{k\in\supp\xi}\vert\xi_k\vert\vert x_k\vert\\
&\leq\sum_{k\in\supp\xi}\vert x_k\vert
=\|P_{\supp\xi}x\|_\1
\leq\gamma_n\|AP_{\supp\xi}x\|_Y.
\end{align*}
That is, $\xi$ belongs to the subdifferential of the convex function $x\mapsto\gamma_n\|AP_{\supp\xi}x\|_Y$ at zero (cf.\ \cite[Lemma~8.31]{Scherzetal09}).
In other words, there is some $\eta\in Y^\ast$ with $\|\eta\|_{Y^*}\leq\gamma_n$ such that
$$\xi=(AP_{\supp\xi})^\ast\eta=P_{\supp\xi}A^\ast\eta.$$
\par
Now assume that (i) and (ii) of Assumption~\ref{as:source} are true.
Fix $n\in\N$ and $x\in\1$ with $\vert\supp x\vert\leq n$. Set $\xi=\sgn x$. Then there is some $\eta\in Y^\ast$ with
$\|\eta\|_{Y^*}\leq\gamma_n$ such that $\xi=P_{\supp\xi}A^\ast\eta$, which is equivalent to
$$\vert\la\xi,\tilde{x}\ra_{\3 \times \1}\vert\leq\gamma_n\|AP_{\supp\xi}\tilde{x}\|_{Y}\quad\text{for all $\tilde{x}\in\1$}$$
(again by \cite[Lemma~8.31]{Scherzetal09}).
Choosing $\tilde{x}=x$ in this inequality we see $\vert\la\xi,\tilde{x}\ra_{\3 \times \1}\vert=\|x\|_\1$
and $P_{\supp \xi}\tilde{x}=x$, which completes the proof.
\end{proof}

From the proposition we deduce that the restricted injectivity property \eqref{eq:rip} can be used to prove convergence rates
for $\ell^1$-regularization even if the solution $x^\dagger$ is not sparse. One only has to add Assumption~\ref{as:source}~(iii).

\begin{proposition}
(a) The condition (\ref{eq:rip}) cannot hold whenever (\ref{eq:strongA}) is valid.
(b) If  (\ref{eq:rip}) holds, then we have
\begin{equation} \label{eq:ge}
0<\gamma_1 \le \gamma_2 \le ... \le \gamma_n \le \gamma_{n+1} \le ...  \to \infty \quad \mbox{as} \quad n \to \infty
\end{equation}
for the constants $\gamma_n$ in (\ref{eq:rip}).
\end{proposition}

\begin{remark}
Part (a) tells us in particular that no compact operator $A$ satisfies the restricted injectivity property (\ref{eq:rip}).
Part (b) implies: If (\ref{eq:strongA}) fails and only (\ref{eq:weakA}) is valid, then due to the ill-posedness of equation (\ref{eq:opeq}) expressed by $\mathcal{R}(A) \not= \overline{\mathcal{R}(A)}$,
which always takes place under the setting of the present paper, we have in case of the validity of the  restricted injectivity property the condition
(\ref{eq:ge})
for the constants $\gamma_n$ in (\ref{eq:rip}). Moreover, taking into account Proposition~\ref{pro:nearequi} this condition also applies to the corresponding constants $\gamma_n$  in Assumption~\ref{as:source}~(i).

\end{remark}

\begin{proof}
The inequality $\gamma_n \le \gamma_{n+1}$ for all $n \in \N$ is evident by inspection of formula (\ref{eq:rip}).
Substituting $x:=e^{(n)}$  with $|\supp(e^{(n)})|=1$ into (\ref{eq:rip}) we have $$\|Ae^{(n)}\|_Y \ge \frac{1}{\gamma_1}\|e^{(n)}\|_{\1}=\frac{1}{\gamma_1}>0.$$
This contradicts (\ref{eq:strongA}) and establishes claim (a).
If $(\gamma_n)_{n \in \NN}$  is bounded then there is some constant $\underline c >0$ such that $\frac{1}{\gamma_n} \ge \underline c$ for all $n \in \N$ in (\ref{eq:rip}).
Consequently, for $\hat x =(\hat x_k)_{k=1}^\infty \in \1$ and $P_n \hat x:=(\hat x_1,...,\hat x_n,0,0,...)$
\[
|\supp (P_n\hat x)| \le n\quad \text{ and }\quad\lim \limits_{n \to \infty}\|\hat x-P_n \hat x\|_{\1}=0
\]
and furthermore
\[
\|AP_n \hat x\|_Y\geq\underline{c}\,\|P_n \hat x\|_\1 \quad\text{for all $n \in \N.$}
\]
Since $A$ is bounded and $\hat x$ was an arbitrary element in $\1$ we even have $$\|A  x\|_Y\geq\underline{c}\,\| x\|_\1 \quad\text{for all $ x \in \1.$}$$
This, however, contradicts the ill-posedness of the problem expressed by $\mathcal{R}(A) \not= \overline{\mathcal{R}(A)}$.
\end{proof}

\section{Examples for the restricted isometry  property}\label{sc:examples}

We present two examples satisfying the restricted injectivity property, hence
Proposition \ref{pro:nearequi} can be applied to them. In fact, the examples 
obey the full uniform uncertainty principle \eqref{eq:uup} as well, with
explicitly known constants.

\begin{example}[Denoising]
We apply (\ref{eq:l1}) to solve a denoising problem based on equation (\ref{eq:opeq}) with $Y=\ell^q(\N), \;1<q \le \infty,$
the noncompact embedding operator $A:\1 \to \ell^q(\N)$ satisfying $\mathcal{R}(A) \not= \overline{\mathcal{R}(A)}$,
and noisy data $y^\delta \in \ell^q(\N)$ (cf.~\cite[Section~5]{FHV15}).
Then we have for $x \in \1$ with $|\supp x| \le n$ the estimates
$$\|Ax\|_{\ell^q(\N)}=\left(\sum _{i=1}^n |x_{\nu_i}|^q\right)^{1/q}  \ge \frac{\sum \limits _{i=1}^n |x_{\nu_i}|}{n^{1-1/q}}=\frac{\|x\|_\1}{n^{1-1/q}} $$
and obviously (\ref{eq:rip}) with $\gamma_n=n^{1-1/q}$ for $1<q<\infty$ and $\gamma_n=n$ in the limit case $q=\infty$.
The constants $\gamma_n$ tend always to infinity as $n \to \infty$.
By the well known inclusion $\ell^1(\NN) \subsetneq \ell^q(\NN)$ with
$\|x\|_{\ell^q(\NN)} \leq \|x\|_{\ell^1(\NN)}$, it follows immediately
that the twosided restricted isometry property \eqref{eq:uup} holds with $\zeta_n=1$
for all $n \in \NN$.
\end{example}

\begin{example}[Restriction operator on the Wiener algebra]
Let $X$ be the Wiener algebra, i.e., the vector space of all complexvalued functions $f$ on the unit circle $\TT$ attaining the form
$$f(t)=\sum\limits_{k\in \ZZ} x_k e^{2\pi ikt}\quad\text{with}\quad x\in\ell^1(\ZZ)$$
equipped with the norm
\[
 \|  f\| _X=\left\| \sum\limits_{k\in \ZZ} x_k e^{2\pi ikt} \right\|_X
:=\sum\limits_{k\in\ZZ}\vert x_k\vert =\|x\|_{\ell^1(\ZZ)}.
\]

Given a measurable subset $E\subseteq\TT$ with positive Lebesgue measure $\vert E\vert > 0$ we consider the operators
\begin{align*}
 L&:\ell^1(\ZZ)\rightarrow X,\quad (x_k)_{k=-\infty}^{+\infty} \mapsto \sum\limits_{k\in \ZZ} x_k e^{2\pi ikt}\\
 \widetilde{A}_E&: X\rightarrow Y := L^\infty(\TT),\quad \widetilde{A}_E f=\chi_Ef.
\end{align*}
Here $L^\infty(\TT)$ is equipped with the sup-norm $\| \cdot\|_{L^\infty(\TT )}$ and $\chi_E$ denotes the characteristic function of $E$ (one on $E$, zero else).
The Turan Lemma in the version of Nazarov (cf.~\cite{Nazarov94}) says:
\begin{lemma}[Nazarov]
Let $n\in \NN, x_1,\dots x_n \in \CC, m_1<\dots <m_n\in \ZZ$ and
\[
 p_x (z) =\sum\limits_{k=1}^n x_k z^{m_k} =\sum\limits_{k=1}^n x_k e^{2\pi im_k t}
\]
 be a trigonometric polynomial on $\TT$ and $E\subseteq \TT$ measurable with $\vert E\vert >0$.
Then
\begin{equation}\label{19}
 \|  p_x\| _\cW =\sum\limits_{k=1}^n \vert x_k\vert \leq \left(\frac{16 \rm e}{\pi \vert E\vert}\right)^{n-1} \sup\limits_{t\in E}\vert p(t)\vert
\leq
 \left(\frac{14}{\vert E\vert}\right)^{n-1}\| \chi_E p\|_{L^\infty(\TT )}
\end{equation}
\end{lemma}
This is a typical uncertainty principle from harmonic analysis.
 Set $A_E=\widetilde{A}_E \circ  L\colon\ell^1(\ZZ)\rightarrow L^\infty(\TT)$ and rewrite (\ref{19}) for any $x \in \ell^1(\ZZ)$
with $|\supp x|\leq n$ as
 \begin{equation}\label{20}
 \|  x\|_{\ell^1(\ZZ)} \leq \left(\frac{14}{\vert E\vert}\right)^{n-1} \|  \widetilde{A}_E (p_x)\|_{L^\infty(\TT )}=\left(\frac{14}{\vert E\vert}\right)^{n-1}
 \|  A_Ex\|_{L^\infty(\TT )}.
 \end{equation}
Thus $A_E$ satisfies the uniform uncertainty principle (and thus in particular the restricted injectivity  property)
\begin{equation}\label{7}
\frac{1}{\gamma_n } \|x\|_{\ell^1(\ZZ)} \leq \|A_E x\|_Y \leq \|  Lx\|_Y =\|  Lx\|_{L^\infty(\TT )} \leq \|Lx\|_X = \|x\|_{\ell^1(\ZZ)}
\end{equation}
for all $ x\in \ell^1$  with  $ |\supp x|\leq n $, where, for alle $n \in \NN$,
\[
\gamma_n =\left(\frac{14}{\vert E\vert}\right)^{n-1} .
\]
Now we choose an arbitrary function $g\colon \TT \to (0,1]$ such that $g \geq \chi_{E}$ where $E\subseteq \TT$ is as above.
Let $M=M_g\colon L^\infty(\TT)\to L^\infty(\TT)$ be the multiplication operator by $g$ and  $A=A_g =M_g \circ L$.
Note that the inverse of $M$ is given by the multiplication by the function $\frac{1}{g}$ which is well defined,
but can be arbitrarily unbounded, depending on the specific choice of $g$. The bounds \eqref{20} and \eqref{7} remain valid if $A_E$ is replaced by $A_g$,
with the same constants. Thus $A_g$ satisfies
a uniform uncertainty principle and in particular a restricted injectivity property.

\begin{remark}
The clou in the Turan lemma is that only the order and not the degree of the polynomial enters!
There exist multi-dimensional analogs of the Wiener Algebra and of the Turan-Nazarov Lemma.
\end{remark}

\end{example}

\section{Outlook}\label{sc:outlook}

Proposition 5.1 establishes a relation between source conditions as used in
\cite{BurFleHof13} and \cite{FleHeg14} and a low-dimensional or restricted injectivity property.
The latter property is part of the famous restricted isometry property \eqref{eq:uup}.
It is used, cf.~e.~g.~\cite{CanRomTao06}, to solve approximately specific inverse problems
for non-injective linear maps. The approximation error can be given explicitly as
a function of the constants $\gamma_n$ and $\zeta_n$, $n \in \NN$ appearing
in \eqref{eq:uup}.

In view of Proposition 5.1.  the following question imposes itself:
Can property (iii) in Assumption~\ref{as:source} be reformulated in a way similar to the restricted injectivity property, i.e.\ without adhering to the adjoint $A^\ast$?
Note that one could slightly generalise property (iii) by requiring only $\|(I-P_{\supp \xi})A^\ast\eta\|_{L^\infty(\TT )}\leq c_n$ for some $c_n\in[0,1)$ independent of $\xi$.
This questions and the application of uniform uncertainty principles to ill posed inverse problems
in general will be pursued in a sequel paper.

\newpage

\section*{Acknowledgement}
J.~Flemming, B.~Hofmann, and I.~Veseli\'c were supported by the German Research Foundation (DFG) under grants FL~832/1-1, HO~1454/8-2, and \linebreak VE 253/6-1 respectively.

%

\end{document}